\documentclass[a4paper,11pt]{amsart}

  \renewcommand\appendix{\par
  \setcounter{section}{0}
  \setcounter{sub
  section}{0}
  \setcounter{figure}{0}
  \setcounter{table}{0}
  \renewcommand\thesection{ Appendix \Alph{section}}
  \renewcommand\thefigure{\Alph{section}\arabic{figure}}
  \renewcommand\thetable{\Alph{section}\arabic{table}}
}

 \usepackage{amsmath, amsthm, amssymb}
\usepackage{amssymb}
\usepackage{graphicx}
\usepackage{algorithmic}
\usepackage{algorithm}

\usepackage{tikz}
\usetikzlibrary{calc}
\usetikzlibrary{patterns}
\tikzstyle{mybox} = [draw=black, fill=white,  thick,
    rectangle, inner sep=10pt, inner ysep=20pt]
\tikzstyle{mybox} = [draw=black, fill=white,  thick,
    rectangle, inner sep=2pt, inner ysep=2pt]

\usepackage[margin=1.2in]{geometry}

\usepackage{setspace}

\newtheorem{corollary}{Corollary}[section]
\newtheorem{definition}{Definition}[section]
\newtheorem{lemma}{Lemma}[section]
\newtheorem{theorem}{Theorem}[section]

\newtheorem{remark}{Remark}[section]
\newtheorem{example}{Example{\rm}}[section]

\begin{document}

\title{On Semi-Invariants of a Matrix}
\author{Amir Jafari\\Amin Najafi Amin}
\address{Department of Mathematical Sciences, Sharif University of Technology, Tehran, Iran}
\email{amirjafa@gmail.com, amin najafiamin@yahoo.com}
\maketitle
\begin{abstract}
For an algebraically closed field $K$ of characteristic zero and a non-singular matrix $A\in \mbox{GL}_n(K)$, a semi-invariant polynomial of $A$ is defined to be a polynomial $p(a)=p(a_0,\dots,a_{n-1})$ with coefficients in $K$ such that $p(Aa)=\lambda p(a)$ for some $\lambda \in K$. In this article, we classify all semi-invariant polynomials of $A$ in terms of a canonically constructed basis that will be made precise in the text.


\end{abstract}


\section{Introduction}\label{sec1}

Let $K$ be a field and $n\ge 1$ be an integer. We let $\mbox{GL}_n(K)$ denote the group of invertible $n\times n$ matrices with entries in $K$. One important problem in
invariant theory (see \cite{H} and \cite{book}) is the problem 
of finding all polynomials $p(a)=p(a_0,\dots,a_{n-1})$ with coefficients in $K$, that are invariant under a given subgroup $G$ of $\mbox{GL}_n(K)$, that is for all $A\in G$
$$p(Aa)=p(a).$$
In this note, we relax this condition and try to find all semi-invariant polynomials $p(a)$ that for all $A\in G$, satisfy
$$p(Aa)=\lambda(A)p(a)$$
where $\lambda(A)\in K$ and is called the multiplier of $p$.
Unlike the invariants that form a $K$-algebra, semi-invariants are closed only under multiplication and the sum of two semi-invariants is a semi-invariant only if their corresponding multipliers are equal.

We only consider the case when $G$ is a cyclic group generated by a single invertible matrix $A$. In this case, we call a semi-invariant polynomial of the cyclic group generated by $A$, a semi-invariant polynomial of $A$. When $K$ is algebraically closed and of characteristic zero, we prove that the semi-invariants, in this case, are in a certain sense, finitely generated and we provide a set of explicitly constructed generators. As was mentioned before, the set of semi-invariants is not an algebra, so this statement must be made precise. In this introduction, we give some key examples to motivate our problem, and then give a precise statement of our result.

Since the homogeneous components of a semi-invariant polynomial are semi-invariant, we may restrict ourselves only to homogeneous polynomials that are simply called forms. The first example reveals the relation between semi-invariant forms and eigenvalues and eigenvectors of $A$.

\begin{example}\label{ex1}

A degree one form $p(a)=c_0a_0+\dots+c_{n-1}a_{n-1}$ is a semi-invariant of $A$ with multiplier $\lambda$, if and only if the row vector $c$, with entries $c_0,\dots, c_{n-1}$ is a (left) eigenvector of $A$ with eigenvalue $\lambda$, that is $cA=\lambda c$. If $A$ is diagonalizable, then we can find $n$ independent eigenvectors that give rise to $n$ independent semi-invariant forms of degree one, $p_1(a),\dots, p_n(a)$ with corresponding multipliers $\lambda_1,\dots, \lambda_n$. It is easy to see that in this case these semi-invariants generate all semi-invariants in the following sense. For a multi-subset (i.e. repetition of elements is allowed) $I$ of $\{1,\dots, n\}$ denote $p_I(a)=\prod_{i\in I}p_i(a)$ and $\lambda_I=\prod_{i\in I}\lambda_i$. Then we have
\begin{theorem}\label{thm11}
Any semi-invariant polynomial $p$ with multiplier $\lambda$ for a diagonalizable matrix $A$ with corresponding semi-invariant forms $p_1,\dots, p_n$ of degree one is of the form
$$\sum c_I p_I$$
where $I$ runs over all multi-subsets of $\{1,\dots, n\}$ with $\lambda_I=\lambda$ and $c_I\in K$ is non-zero only for finitely many $I$.
\end{theorem}

\end{example}
When $A$ is not diagonalizable, these linear forms are not enough to generate all semi-invariant forms of $A$ in general, and some interesting semi-invariant quadratic or cubic forms may be needed. We now explain them. Let $J_{n,\lambda}$ denote the $n\times n$ (multiplicative) Jordan block, with $\lambda$ on its main diagonal, $\lambda$ on its off-diagonal right below it, and zero everywhere else. For example $J_{4, \lambda}$ is the following matrix.
\[J_{4,\lambda}= \left[ \begin{array}{cccc}
\lambda&0&0&0\\
\lambda&\lambda&0&0\\
0&\lambda&\lambda&0\\
0&0&\lambda&\lambda
\end{array}\right]\]
\begin{example}\label{ex2}
 If $A=J_{3, \lambda}$ then other than the semi-invariant polynomial $p_1(a_0,a_1,a_2)=a_0$ with multiplier $\lambda$ obtained from the only eigenvector of $A$ as in Example \ref{ex1}, the polynomial
$$p_2(a_0,a_1,a_2)=-a_1^2+2a_2a_0+a_1a_0$$
is a semi-invariant form with multiplier $\lambda^2$.
\end{example}
\begin{example}\label{ex3}
 If $A=J_{4, \lambda}$ then other than the linear and quadratic semi-invariant forms $p_1=a_0$ and $p_2=-a_1^2+2a_2a_0+a_1a_0$ from Example \ref{ex1} and Example \ref{ex2}, the cubic form
$$p_3(a_0,a_1,a_2,a_3)=-a_1^3+3a_2a_1a_0-3a_3a_0^2-2a_2a_0^2+a_1^2a_0$$
is a semi-invariant form with multiplier $\lambda^3$.
\end{example}
\begin{example}\label{ex4}
 If $A=\mbox{diag}(J_{2, \lambda_1},J_{2, \lambda_2})$ is a square matrix of size $4$ with two Jordan blocks, then other than the linear semi-invariant polynomials $p_1=a_0$ and $p_2=a_2$ with multipliers $\lambda_1$ and $\lambda_2$, the quadratic form $p_3(a_0,a_1,a_2,a_3)=a_0a_3-a_1a_2$ is a semi-invariant form with multiplier $\lambda_1\lambda_2$.
 
 \end{example}
\begin{example}\label{ex5}
 In Example \ref{ex3} above, any polynomial $p$ of the form
$$\sum c_{i_1,i_2,i_3}p_1^{i_1}p_2^{i_2}p_3^{i_3}$$
where $i_1+2i_2+3i_3$ is a fixed integer $k$ is a semi-invariant form of degree $k$ and multiplier $\lambda^k$. One might conjecture that these are all such semi-invariant forms for $A$, which is in fact wrong. The following degree $4$ form $p$ given by
$$
-3a_2^2a_1^2+6a_3a_1^3+8a_2^3a_0-18a_3a_2a_1a_0+9a_3^2a_0^2
$$
$$
+3a_2a_1^3-6a_2^2a_1a_0-9a_3a_1^2a_0+18a_3a_2a_0^2
$$
$$
-5a_2a_1^2a_0+8a_2^2a_0^2+3a_3a_1a_0^2
$$
$$
+2a_2a_1a_0^2
$$
is a semi-invariant form with multiplier $\lambda^4$. However it is easy to see that it can not be written as a polynomial in terms of $p_1,p_2$ and $p_3$ . Nevertheless, it can be written as a rational function
$$p=\frac{p_2^3-p_1p_2p_3+p_3^2}{a_0^2}.$$
\end{example}

Our goal is to generalize these examples. Let $A$ be an invertible matrix, given in its Jordan normal form
$$A=\mbox{diag}(J_{n_1,\lambda_1},\dots, J_{n_k,\lambda_k})$$
with $n_1\ge \dots \ge n_l\ge 2> n_{l+1}=\dots=n_k=1$, where we may assume that $l>0$, since the case of a diagonalizable matrix was handled in Theorem \ref{thm11}. We take the variables of our polynomials as 
$$(a_{0,1},\dots, a_{n_1-1,1},a_{0,2},\dots, a_{n_2-1,2},\dots, a_{0,k},\dots, a_{n_k-1,k}).$$

 Then from Example \ref{ex1}. we have $k$ semi-invariant linear forms
 $a_{0,1}, a_{0,2},\dots, a_{0,k}$, with multipliers $\lambda_1,\dots,\lambda_k$ respectively. From a generalization of Examples \ref{ex2} and \ref{ex3}, for each $J_{n_i,\lambda_i}$ with $n_i\ge 3$, we construct an extra $n_i-2$ quadratic and cubic semi-invariant forms with multipliers $\lambda_i^2$ and $\lambda_i^3$ respectively. Also from a generalization of Example \ref{ex4}, for each pair $J_{n_i,\lambda_i}$ and $J_{n_{i+1},\lambda_{i+1}}$ for $i=1,\dots, l-1$, we construct $l-1$ quadratic semi-invariant forms $a_{0,i}a_{1,i+1}-a_{1,i}a_{0,i+1}$ with multipliers $\lambda_i\lambda_{i+1}$. So together, we have $n-1$ linear, quadratic or cubic forms, say $p_1,\dots, p_{n-1}$ with multipliers $\mu_1,\dots, \mu_{n-1}$ respectively. Our main result is the following theorem.
 
 \begin{theorem} \label{main}
 If the field $K$ is of characteristic zero, then the above mentioned semi-invariant forms $p_1,\dots, p_{n-1}$ generate all semi-invariant forms in the following sense. Any semi-invariant polynomial for $A$ with multiplier $\lambda$ can be uniquely expressed as a rational function
 $$\frac{\sum_{I} c_I p_I}{\prod_{i=1}^{l} a_{0,i}^{m_i}}$$
 where $I$ runs over all multi-subsets of $\{1,\dots, n-1\}$ such that $\prod_{i\in I}\mu_i=\lambda\prod_{i=1}^l \lambda_i^{m_i}$, $p_I=\prod_{i\in I}p_i$ and $c_I\in K$ is non-zero only for finitely many $I$.
 \end{theorem}

 Let us briefly explain the rich history of the problem studied in this article. Invariant theory, started in the middle of the nineteenth century, from the works of Boole and Cayley, see \cite{Bo} and \cite{Ca}. The classical invariant theory deals with algebraic expressions in terms of coefficients of a form (say a binary form) that remain invariant under the action of the group of general linear transformations. As an example, for $a_0X_1^2+a_1X_1X_2+a_2X_2^2$ the discriminant $\Delta=a_1^2-4a_0a_2$ is an invariant. This point is explained in more detail in section 3. In modern language, the invariant theory deals with the following problem. Assume a group $G$ acts linearly on a finite-dimensional vector space $V$ over a field $K$. Then we have an action of $G$ on the space of polynomials $K[V]$ and we want to study the $G$ invariant polynomials $K[V]^G$. For example, Hilbert proved that if $G$ is a reductive group then this algebra is finitely generated. The famous 14th problem of Hilbert asks about the validity of this statement for other groups. It was first Nagata who in 1959 \cite{N} gave a counterexample for this question. It was proved by Weitzenb\"ock in 1932 \cite{W} that if $G$ is the additive group of complex numbers, then finite dimensionality is true. This theorem is valid for any field of characteristic zero and algebraically closed. In our set up, it will imply that for the Jordan bock $J=J_{n,1}$ or more generally for $\mbox{diag}(J_{n_1,1},\dots, J_{n_k,1})$ the space of invariants is finitely generated. This seems to be a better result than the main theorem above, where we have allowed certain denominators. However, the problem is that there is no algorithm to find such a finite set of generators as of now. Also, it is worthwhile to mention the following celebrated theorem of Popov \cite{P} that says if $G$ is a non-reductive group then there is a linear action of $G$ on certain finite-dimensional vector space $V$ over a field $K$ such that $K[V]^G$ is not finitely generated. This for instance shows that the above-mentioned theorem of Weitzenb\"ock is not true over fields of positive characteristic. It is extremely interesting to generalize the results of this paper to the case of fields with positive characteristics. Finally, for a modern approach to invariant theory, we encourage the reader to see the references \cite{CD} and \cite{PP}.

 \begin{remark} We make the following observations.
 
 \begin{enumerate}
 \item Note that a polynomial is semi-invariant for the (multiplicative) Jordan block $J_{n,\lambda}$ if and only if it is invariant for $J_{n,1}$, which we usually denote by $J_n$.

 \item If instead of a multiplicative Jordan block $J_{n,\lambda}$ as above, we use an additive Jordan block $J'_{n,\lambda}$ with $\lambda$ on its main diagonal and $1$ on its off diagonal, below the main diagonal and zero everywhere else, then $p(a_0,a_1,\dots,a_{n-1})$ is a semi-invariant for $J_{n,\lambda}$ if and only if $q(a_0,a_1,\dots,a_n):=p(a_0,\lambda a_1,\lambda^2a_2,\dots, \lambda^{n-1} a_{n-1})$ is a semi-invariant for $J'_{n,\lambda}$.
 \item In general we may see easily that $p(a_{0,1},\dots, a_{{n_1-1},1},\dots, a_{0,k},\dots, a_{n_k-1,k})$ is a semi invariant for the matrix $A=\mbox{diag}(J_{n_1,\lambda_1},\dots, J_{n_k,\lambda_k})$ if and only if $$q:=p(a_{0,1},\lambda_1a_{1,1},\dots, \lambda_1^{n_1-1}a_{{n_1-1},1},\dots, a_{0,k},\dots, \lambda_k^{n_k-1}a_{n_k-1,k})$$ is a semi-invariant for the matrix $A'=\mbox{diag}(J'_{n_1,\lambda_1},\dots, J'_{n_k,\lambda_k})$.
 \end{enumerate}
 \end{remark}
 \section{Main results}\label{sec2}
  
 In this section, a canonical basis for semi-invariants of a non-singular matrix $A$ given in its Jordan normal form is constructed. First, we need a few lemmas and conventions.
 \\
 \begin{definition}
 For a monomial $a_0^{i_0}\dots a_{n-1}^{i_{n-1}}$, its degree is $\sum_{k=0}^{n-1} i_k$ and its weight is $\sum_{k=0}^{n-1}ki_k$. Any polynomial can be decomposed into its degree homogeneous components and its weight homogeneous components.
 \end{definition}
 \begin{definition}
 We define a differential operator $D$ on the polynomial algebra $K[a_0,a_1,\dots]$ by
 $$D=\sum_{i=1}^{\infty} a_{i-1}\frac{\partial}{\partial a_i}.$$
 \end{definition}
 \begin{remark}
 The operator $D$ sends a degree homogeneous polynomial of degree $d$ to a degree homogeneous polynomial of degree $d$. It sends a weight homogeneous polynomial of weight $w$ to a weight homogeneous polynomial of weight $w-1$. Hence for any polynomial $p$ of maximum weight $w$, one has $D^{w+1}(p)=0$.
 \end{remark}
 
 \begin{lemma}
 Let $p\in K[a_0,\dots, a_{n-1}]$ and $J_n=J_{n,1}$ be the (multiplicative) Jordan block of size $n$ and eigenvalue $1$. Then 
 $$D(p(J_na))=(Dp)(J_na)$$
 that is the two operators $D$ and $p(a)\rightarrow p(J_na)$ commute. This implies that if $p$ is $J_n$-invariant then so is $D(p)$.
 \end{lemma}
 \begin{proof} By chain rule
 $$\frac{\partial}{\partial a_i} p(a_0,a_1+a_0,\dots, a_{n-1}+a_{n-2})=\left((\frac{\partial}{\partial a_i}+\frac{\partial}{\partial a_{i+1}})p\right)(J_na)$$
 where for $i=n-1$, the term corresponding to $a_{i+1}$ is omitted. Therefore
 $$D(p(J_na))=\sum_{i=1}^{n-1}(a_{i-1}+a_{i-2})(\frac{\partial}{\partial a_i} p) (J_na)=(Dp)(J_na)$$
 Here for $i=1$, the term $a_{i-2}$ is omitted.
 \end{proof}
 
 \begin{lemma} \label{JJ}
 If $p\in K[a_0,\dots, a_{n-1}]$ is a $J_n$-invariant polynomial, and $p_w$ is the weight homogeneous component of $p$ with highest weight $w$, then $D(p_w)=0$.
 \end{lemma}
 \begin{proof} If $p=a_0^{i_0}\dots a_{n-1}^{i_{n-1}}$ is a monomial of weight $w$, then 
 $$p(J_na)=a_0^{i_0}(a_1+a_0)^{i_1}\dots (a_{n-1}+a_{n-2})^{i_{n-1}}.$$
 It is now clear that the components of $p(J_na)$ are of weights $\le w$. The component of weight $w$ is $p$ and the component of weight $w-1$ is obtained if from one of the parenthesis say $(a_k+a_{k-1})^{i_k}$ we choose $a_k^{i_k-1}a_{k-1}$ whose coefficient is $i_k$, so it is like taking $a_{k-1}\frac{\partial}{\partial a_k}$ of $p$, this shows that the component of weight $w-1$ is $D(p)$. Now if $p=p_w+p_{w-1}+\dots$, then $p(J_na)=p+D(p_w)+\mbox{terms of weights smaller than}\:\: w-1$. Since $p(J_na)=p(a)$, hence we conclude that $D(p_w)=0$.
 \end{proof}
 
 \begin{definition}
 A polynomial $p$ in $K[a_0,a_1,\dots]$ with $Dp=0$ is called a $U$ invariant. The subset of $U$ invariant polynomials form a sub-algebra of $K[a_0,a_1,\dots]$ denoted by $K[a_0,a_1,\dots]^U$.
 \end{definition}

\begin{lemma}\label{lem1}
Let $A=J_{n+1,\lambda}$ and $n=2m>0$ be an even number. The polynomial 
$$p_{n}=\sum_{i=0}^{m} \sum_{j=m}^{2m-i}(-1)^iA_{i,j} a_ia_{j}$$
where
$$A_{i,j}={m-i\choose j-m}+{m-i-1\choose j-m-1}$$
is a semi-invariant polynomial with multiplier $\lambda^2$. The only term with $a_{n}$ in this form is $2a_{n}a_0$.
\end{lemma}
Note that by convention $n\choose k$ is zero if $k<0$ or $n<k$ and is $1$ if $n=k=0$.
\begin{proof}
It is enough to show that $p_{n}$ is a $J=J_{n+1,1}$-invariant. Note that 
$$p_{n}(Ja)=\sum_{j=m}^{2m} A_{0,j}a_0(a_j+a_{j-1})+\sum_{i=1}^m \sum_{j=m}^{2m-i} (-1)^i A_{i,j}(a_i+a_{i-1})(a_j+a_{j-1})$$
Now to show the invariance of $p_{n}$ for $J$, we need to show that for $m<j<2m$
$$A_{0,j}=A_{1,j}+A_{1,j-1}$$
for $j=m$
$$A_{0,m}=A_{1,m}$$
and for $j=2m$
$$A_{0,2m}=A_{1,2m-1}.$$

These will show that the terms with $a_0$ in both of $p_{n}(Ja)$ and $p_{n}(a)$ are the same. To show this for other terms of the form $a_ia_j$ with $1\le i< m$ and $m\le j<2m-i$, we need to show that 
$$A_{i,j+1}=A_{i+1,j}+A_{i+1,j+1}$$
All of these trivially follow with the aid of Pascal's identity. If $i=m$ then $j=m$ then all the terms are zero. Also if $j=2m-i$, again all the terms are zero. Also the only term with $a_{n}$ appears if $j=n=2m$, and hence $i=0$ and the coefficient turns out to be $2$.

\end{proof}

\begin{remark}
The quadratic forms $p_n$, for even $n\ge 2$ can be defined inductively as follows:
$$p_n(a_0,\dots, a_n)=-p_{n-2}(a_1,\dots, a_{n-1})+a_0q_n(a_1,\dots, a_n)$$
where $q_2(a_1,a_2)=a_1+2a_2$ and
$$q_n(a_1,\dots, a_n)=q_{n-2}(a_2,\dots, a_{n-1})+q_{n-2}(a_3,\dots, a_n).$$
We also set $p_0(a_0)=a_0^2$.
\end{remark}

\begin{example}
We list here two more quadratic semi-invariant forms for the Jordan blocks $J_{5,\lambda}$ and $J_{7,\lambda}$. 

\begin{eqnarray}
p_4(a)&=&(a_2^2-2a_3a_1+2a_4a_0)+(-a_2a_1+3a_3a_0)+a_2a_0\nonumber\\
p_6(a)&=&(-a_3^2+2a_4a_2-2a_5a_1+2a_6a_0)+(a_3a_2-3a_4a_1+5a_5a_0)\nonumber\\
&+&(-a_3a_1+4a_4a_0)+a_3a_0\nonumber
\end{eqnarray}

An amusing fact is that $Dp_4=p_2$ and $Dp_6=p_4$. It will be shown in Lemma 2.5 that for $n\ge 4$ $Dp_n=p_{n-2}$.

\begin{lemma}\label{lem2}
 Let $A=J_{n+1,\lambda}$ and $n\ge 3$ be an odd integer. Let $m=\frac{n+1}{2}$. The cubic form
 $$p_{n}=a_1p_{n-1}+a_0g_{n}$$
 where
 $$g_{n}=\sum_{i=0}^{m-1}\:\:\sum_{j=m}^{2m-1-i} (-1)^{i-1}(j-i){m-i-1\choose j-m}a_ia_j.$$
 is a semi-invariant form for $A$ with multiplier $\lambda^3$. Also the only term with $a_{n}$ in $p_{n}$ is $-na_{n}a_0^2$.
\end{lemma}
\begin{proof}
Since we know that $p_{n-1}$ is $J_{n}$ invariant, we need to show that 
$$g_{n}(J_{n+1}a)=g_{n}(a)-p_{n-1}(a)$$
Now $g_{n}(J_{n+1}a)-g_{n}(a)$ is equal to 
$$-\sum_{j=m}^{2m-1}j{m-1\choose j-m}a_0a_{j-1}-\sum_{i=1}^{m-1}\sum_{j=m}^{2m-1-i}(-1)^i(j-i){m-i-1\choose j-m}(a_{i-1}a_j+a_{i-1}a_{j-1}+a_ia_{j-1})$$
The coefficient of $a_0a_{j-1}$ is 
$$-j{m-1\choose j-m}+(j-1){m-2\choose j-m}+(j-2){m-2\choose j-m-1}.$$
By Pascal's identity, this is equal to $-{m-2\choose j-m}-2{m-2\choose j-m-1}$ which is equal to $-{m-1\choose j-m}-{m-2\choose j-m-1}$. This is the same as the coefficient of $a_0a_{j-1}$ in $-p_{n-1}$. Note that $\frac{n-1}{2}=m-1$. A similar calculation, shows that the coefficient of $a_ia_{j-1}$ is
$$(-1)^{i-1}\left((j-i){m-i-1\choose j-m}-(j-i-1){m-i-2\choose j-m}-(i-j-2){m-i-2\choose j-m-1}\right)$$
which is equal to 
$$(-1)^{i-1}\left({m-i-2\choose j-m}+2{m-i-2\choose j-m-1}\right)=(-1)^{i-1}\left({m-i-1\choose j-m}+{m-i-2\choose j-m}\right)$$
and this is exactly the coefficient of $a_ia_{j-1}$ in $-p_{n-1}$.

This proves the invariance. Also the only term with $a_{n}$ appears in $a_0g_{n}$. Here $j=n=2m-1$ and $i=0$, and the coefficient turns out to be $-n$.
\end{proof}
\begin{remark}\label{g}
Using the explicit definitions of $g_n$ and $p_{n+1}$, for an odd value of $n$, one sees that
$$g_n=W(p_{n+1})-(n+1)p_{n+1}$$
where $W$ is an operator on $K[a_0,a_1,\dots]$ that sends a polynomial $f$ of weight $w$ to $wf$, and is extended linearly.
\end{remark}

\end{example}
\begin{example} We also list two more cubic semi-invariant forms for the Jordan blocks $J_{6,\lambda}$ and $J_{8,\lambda}$.

\begin{eqnarray}
p_5(a)&=&(a_2^2a_1-2a_3a_1^2-a_3a_2a_0+5a_4a_1a_0-5a_5a_0^2)+(-a_2a_1^2+5a_3a_1a_0-8a_4a_0^2)\nonumber\\
&+&(a_2a_1a_0-3a_3a_0^2)\nonumber\\
p_7(a)&=&(-a_3^2a_1+2a_4a_2a_1-2a_5a_1^2+a_4a_3a_0-3a_5a_2a_0+7a_6a_1a_0-7a_7a_0^2)\nonumber\\
&+&(a_3a_2a_1-3a_4a_1^2-2a_4a_2a_0+13a_5a_1a_0-18a_6a_0^2)\nonumber\\
&+&(-a_3a_1^2+7a_4a_1a_0-15a_5a_0^2)+(a_3a_1a_0-4a_4a_0^2)\nonumber
\end{eqnarray}

An amusing fact is that $Dp_5=p_3$ and more generally it will be shown below that in genral $Dp_n=p_{n-2}$. It is interesting that $Dp_3=0$.
\end{example}

\begin{lemma}
If $n\ge 4$ then $Dp_n=p_{n-2}$.
\end{lemma}
\begin{proof} We first show that if $n\ge 5$ and is odd then $Dg_{n}=g_{n-2}-p_{n-1}$. 
This follows from the explicit representation of $g_n$ as a polynomial
$$g_{n}(a)=\sum_{i=0}^{m-1}\sum_{j=m}^{2m-1-i} (-1)^{i-1}(j-i){m-i-1\choose j-m}a_ia_j.$$
where $m=\frac{n+1}{2}$. The details are left to the reader.
According to the proof of Lemma 2.4 we have
$$g_n(Ja)=g_n(a)-p_{n-1}(a)$$
If we take $D$ from both sides and use Lemma 2.1 and the above calculation, we get
$$g_{n-2}(Ja)-p_{n-1}(Ja)=g_{n-2}(a)-p_{n-1}(a)-Dp_{n-1}.$$
Now since $p_{n-1}$ is $J$-invariant, we deduce that $Dp_{n-1}=p_{n-3}$ for $n\ge 5$ and odd. This proves the Lemma for even values of $n$.
If $n\ge 5$ and is odd then
$$Dp_{n}=a_0p_{n-1}+a_1p_{n-3}+a_0(g_{n-2}-p_{n-1})=a_1p_{n-3}+a_0g_{n-2}=p_{n-2}.$$
And the lemma is proved completely. 
\end{proof}
\begin{lemma}\label{lem3}
Let $A=\mbox{diag}(J_{n_1,\lambda_1},\dots, J_{n_k,\lambda_k})$ be a square matrix in Jordan normal form. With $n_1\ge \dots\ge n_l>1=n_{l+1}=\dots=n_k$. Then for $i=1,\dots l-1$ the quadratic forms
$$p_{n-l+i}:=a_{0,i}a_{1,i+1}-a_{0,i+1}a_{1,i}$$
are semi-invariant forms with multiplier $\lambda_i\lambda_{i+1}$.
\end{lemma}
\begin{proof}
Note that 
$$(\lambda_i a_{0,i})(\lambda_{i+1}a_{0,i+1}+\lambda_{i+1} a_{1,i+1})-(\lambda_{i+1}a_{0,i+1})(\lambda_i a_{0,i}+\lambda_i a_{1,i})= \lambda_i\lambda_{i+1}(a_{0,i}a_{1,i+1}-a_{0,i+1}a_{1,i})$$

\end{proof}
\begin{remark}
These polynomials are extensions of Examples \ref{ex1}, \ref{ex2} and \ref{ex3} in the introduction.
\end{remark}

The ground field $K$ is assumed to be algebraically closed and of characteristic zero.

\begin{theorem}\label{thm1}
If $A$ is a Jordan block $J_{n+1,\lambda}$ with $\lambda\ne 0$, then if $n=0$ or $n=1$ then all semi-invariant forms are of the form $a_0^i$. If $n\ge 2$, then we have semi-invariant forms $p_2(a_0,a_1,a_2)$, $p_3(a_0,a_1,a_2, a_3),
\dots, p_{n}(a_0,\dots, a_{n})$ of alternative degrees $2$ and $3$ constructed from lemma \ref{lem1} and lemma \ref{lem2} together with $p_1=a_0$ form a basis in the sense that any semi-invariant form $p$ is uniquely written as 
$$\frac{\sum c_{i_1,\dots, i_{n}}p_1^{i_1}\dots p_{n}^{i_{n}}}{a_0^m}$$
where $i_1+2i_2+3i_3+2i_4+\dots$ is fixed and $m\ge 0$ is an integer. 
\end{theorem}
\begin{proof}
The case $n=0$, is trivial. Now let $n=1$. Assume that $f(a_0,a_1)$ is a semi-invariant form of degree $m$, with $f(a_0, a_1+a_0)=f(a_0,a_1)$. Hence $f(a_0,a_1)=f(a_0,a_1+ka_0)$ for all integers $k\ge 0$ so if we fix $a_0\ne 0$, the polynomial is a constant (since it takes infinitely many equal values), so $f(a_0,a_1)=c a_0^m$.

Now we prove the theorem by induction on $n$. It is easy to check that $p_1,\dots,p_{n}$ are algebraically independent. Note that by the induction hypothesis $p_1,\dots, p_{n-1}$ are algebraically independent and with variables $a_0,\dots, a_{n-1}$. Now $p_{n}$ has variable $a_{n}$ with non-zero coefficient (this follows from the fact that $K$ is of characteristic zero) and hence $p_1,\dots,p_{n}$ are algebraically independent. Let $p(a_0,\dots,a_n)$ be a semi-invariant polynomial, write it as
$$p=\sum_{i=0}^m h_i(a_0,\dots,a_{n-1})a_{n}^i.$$
We prove the theorem by another induction on $m$. If $m=0$, then $p$ is a semi-invariant for the $J_{n,\lambda}$ and hence by the induction assumption for $n-1$, it has the required representation in term of $p_1,\dots,p_{n-1}$.

 Now assume that the desired representation is proved for semi-invariant forms with $a_{n}$-degree less than $m$ and we want to prove it for semi-invariant forms of $a_{n}$-degree equals to $m$. By comparing the highest power of $a_{n}$ in both sides of $p(J a)=p(a)$ it follows that $h_m(J a)= h_m(a)$, where here $x=(a_0,\dots, a_{n-1})$. Again, by induction hypothesis for $n-1$, $h_m$ has the required representation in terms of $p_1,\dots, p_{n-1}$. Note that by the construction of $p_{n}(a)$ in Lemma \ref{lem1} and Lemma \ref{lem2}, one has
$$p_{n}(a)=ca_{n}a_0^r+s(a_0,\dots ,a_{n-1})$$
where $r=1$ if $n$ is even and $r=2$ if $n$ is odd, $c$ is a non-zero element (since characteristic is zero) and $p_{n}(J a)= p_{n}(a)$. In fact,
\[c=\begin{cases}
 -n\quad \mbox{for odd}\: n\\
 2\quad\quad \mbox{for even}\: n
 \end{cases}
 \]

It follows that
$$q(a)=(ca_0^r)^mp(a)-h_m(a)(p_{n}(a))^m$$
is an invariant form for $J_{n+1}$ and $a_{n}$-degree less than $m$. So by the induction hypothesis, $q(a)$ has the desired representation in terms of $p_1,\dots, p_{n}$, since we saw that $h_m(a)$ also has such a representation, solving for $p$ will give us the desired representation for $p$. The theorem is proved.
\end{proof}
\begin{remark} There are many more examples of basis for $A$ besides the one given in Theorem \ref{thm1}. In fact any set of semi-invariants $f_i\in K[a_0,\dots, a_i]$ with only one term with $a_i$ and of the form $c_i a_0^r a_i$ with $c_i\ne 0$ and $f_1=a_0$ is a basis in the above sense. The same exact proof as above works.
\end{remark}

A similar theorem can be stated for the algebra of $U$ invariant polynomials in $K[a_0,\dots, a_n]$.

\begin{theorem}\label{U}
Let $q_1=a_0$, and $q_i\in K[a_0,\dots, a_i]$ for $i=2,\dots, n$ be $U$ invariant homogeneous polynomials and with only one term with $a_i$ and of the form $c_i a_0^r a_i$ and $c_i\ne 0$. Then, any $U$ invariant polynomial is of the form 
$$\frac{P(q_1,\dots, q_n)}{a_0^m}.$$
for a polynomial $P$.
\end{theorem}
\begin{proof} This is exactly like the previous theorem. We only need to replace $J$-invariant with $U$-invariant. For example if $q$ is $U$ invariant and 
$$q=\sum_{i=0}^m h_i(a_0,\dots, a_{n-1})a_n^i$$
then $h_m$ is $U$ invariant, this follows if we apply $D$ to both sides and look at the coefficient of $a_n^m$. Hence $h_m$ has the desired representation in terms of $q_1,\dots, q_{n-1}$ by induction. Now
$$Q(a)=(c_n a_0^r)^m q(a)-h_m(a)(q_n(a))^m$$
is $U$ invariant and has a $a_n$-degree less than $m$, hence similar as before the theorem follows by the induction hypothesis. 
\end{proof}

\begin{remark}\label{procesi}
In \cite{KP}, Kraft and Procesi have given another set of U invariants of equal weight and degree that satisfy the condition of the above theorem. They are given by the following formula.

$$(-1)^kC_k=(1-k)\frac{a_1^k}{k!}+\sum_{j=2}^k \frac{(-1)^j}{(k-j)!} a_0^{j-1}a_1^{k-j}a_j$$
This basis can be easily calculated as polynomials in terms of the U invariants $q_k$ that are the weight $k$ part of $p_k$. The description is inductive and is as follows.
$$2C_2=q_2$$
$$3C_3=q_3$$
$$2C_4=a_0^2q_4-C_2^2$$
$$5C_5=-a_0^2q_5-C_2C_3$$
and more generally if $n=2k$ is even then 
$$2(-1)^kC_n= a_0^{n-2}q_n-C_k^2- \sum_{j=2}^{k-1} (-1)^{j}C_jC_{n-j}$$
and if $n=2k-1$ is odd then
$$n(-1)^kC_n= a_0^{n-3}q_k-\sum_{j=2}^{k-1} (-1)^{j+k}(k-2j)C_jC_{k-j}.$$
In fact polynomials of $q_1,\dots, q_n$, produce a much larger space of $U$ invariants than polynomials of $C_1,\dots, C_n$.
\end{remark}

\begin{remark}
In studying the semi-invariants of a matrix, we arrived at the definition of the polynomials $p_n$. We were quite surprised to see that the weight $n$ parts of $p_n$ appeared (up to a constant factor and changing his variables $a_i$ to $i!a_i$) in Lecture XIX (page 61) of Hilbert's classical book \cite{Hil}. He used them to construct covariants of degree 2 and degree 3 of binary forms. This connection will be explained more in section 3. He also had a theorem similar to the theorem \ref{U} above for the generation of all covariants in Lecture XX in loc. cit. 
\end{remark}

\begin{theorem}\label{thm2}
Let $A=\mbox{diag}(J_{n_1,\lambda_1},\dots, J_{n_k,\lambda_k})$ be a non-singular square matrix in Jordan normal form, with $l>0$ blocks of size $>1$, say $J_{n_{i_j},\lambda_{i_j}}$ for $j=1,\dots, l$. Then by the previous theorem each Jordan block of size $n_i>1$ will give $n_i-1$ semi-invariant polynomials and so together we get $n-l$ semi-invariant polynomials $p_1,\dots, p_{n-l}$ and then by using lemma \ref{lem3} we construct $l-1$ quadratic semi-invariant forms $p_{n-l+1},\dots, p_{n-1}$. Let $p_i(Aa)=\mu_i p(a)$ for $i=1,\dots, n-1$. Then, any semi-invariant polynomial $p(a)$ with $p(Aa)=\alpha p(a)$ can be written uniquely as 
$$\frac{\sum c_{i_1,\dots, i_{n-1}}p_1^{i_1}\dots p_{n-1}^{i_{n-1}}}{\prod_{j=1}^la_{0,j}^{m_j}}$$
where $m_j\ge 0$ are integers and for all $i_1,\dots, i_{n-1}$ with non zero $c_{i_1,\dots,i_{n-1}}$ we must have
$\mu_1^{i_1}\dots \mu_{n-1}^{i_{n-1}}=\alpha \lambda_{i_1}^{m_1}\dots \lambda_{i_l}^{m_l}$. 
\end{theorem}
\begin{proof}
The proof of algebraic independence of $p_1,\dots, p_{n-1}$ is by an application of the Jacobian criterion, see \cite{L}, chapter one, section 11.4. It says that 
\\
\\
{\it{A set of $m$ polynomials $p_1,\dots, p_m$ in $K[a_1,\dots, a_n]$ with coefficients in a field $K$ of characteristic zero and $m\le n$ is algebraically independent, if and only if, the $m\times n$ Jacobian matrix $[\frac{\partial p_i}{\partial a_j}]$ as a matrix in the field $L=K(a_1,\dots, a_n)$ is of rank $m$.}}
\\
\\
 Now, using this fact, we prove algebraic independence. First of all, we showed in Theorem \ref{thm1} that the forms obtained from each Jordan block are algebraically independent. It is also clear that forms from different blocks since they contain a disjoint set of variables are algebraically independent. The only issue, if any, might happen due to the existence of forms arising from Lemma \ref{lem3} by combining two consecutive Jordan blocks of size $>1$, we call these mixed forms. So for the sake of contradiction, assume that a linear combination of the rows of the Jacobian matrix above is zero. It must contain a row, contributed from one of the mixed forms. Let $p_i$ be the first (with the order given by the variables) such mixed form whose corresponding row appears in the assumed linear relation amongst the rows of the Jacobian matrix. Assume that it combines the two consecutive blocks $J_{n_j,\lambda_j}$ and $J_{n_{j'},\lambda_{j'}}$ of size $>1$. The entry corresponding to $\frac{\partial p_i}{\partial a_{1,j}}$, can only be cancelled by the rows corresponding to the quadratic or cubic forms for the block $J_{n_{j'},\lambda_{j'}}$. Now let $p_r$ be the first of these forms appearing in the assumed linear relation. Let the variable in this form with the largest index be $a_{s,j'}$. The term $\frac{\partial p_r}{\partial a_{s,j'}}$ is a monomial with $a_{0,j'}$ and can not be canceled by any other row of the Jacobian matrix. This contradiction, proves the algebraic independence of $p_1,\dots, p_{n-1}$.

Let us now, prove the desired representation of a semi-invariant form $p(a)$ with multiplier $\lambda$ by an induction on $n$, the size of matrix $A$. Write
$$p(a)=\sum_{i=0}^m h_i(a_{0,1},\dots,a_{n_k-2,k})a_{n_k-1,k}^i.$$
 With a cyclic shift of variables, without loss of generality, we may assume that $n_k>1$. If $m=0$, then since $p$ is with $n-1$ variables and is semi-invariant for the sub-matrix $A'$ obtained by removing the last row and column of $A$, the desired representation follows by the induction hypothesis. Now we prove the theorem with another induction on $m$.
 If $n_k>2$, then the same proof as before using $p_{n_k-1}$ will reduce $m$ and finish the proof by the induction hypothesis on $m$. If $n_k=2$, and all $n_1=\dots=n_{k-1}=1$ then the statement of the theorem becomes trivial, using a similar technique used in the previous theorem for $n=2$. Finally if $n_k=2$ and the index $i<k$ that $n_i>1$ exists, let $i$ be the largest such index. Use $r(a)=a_{0,i}a_{1,k}- a_{1,i}a_{0,k}$ to reduce the power of $a_{1,k}$ in $p(a)$ as follows. By comparing the highest power of $a_{1,k}$ in $p(Aa)$ and $p(a)$, it follows that $h_m(A'a)=h_m(a)$. Here $A'$ is obtained from $A$ by removing its last row and column. Now define
$$q(a)=a_{0,i}^mp(a)-r(a)^mh_m(a).$$
Then $q(a)$ is an invariant form whose $a_{1,k}$-degree is less than $m$. Therefore, since by the hypothesis of induction both $h_m(a)$ and this form have the desire representation in terms of $p_1,\dots, p_{n-1}$, hence $p$ can be represented as claimed.
\end{proof}

We mention the following elementary observation. If two matrices $A$ and $B$ are similar, i.e. if there is an invertible matrix $S$ such that $A=SBS^{-1}$, then $p(a)$ is a semi-invariant for $B$ with multiplier $\lambda$ if and only if $q(a)=p(S^{-1}a)$ is a semi-invariant for $A$ with multiplier $\lambda$. This is because
$$q(Aa)=p(S^{-1}Aa)=p(BS^{-1}a)=p(S^{-1}a)=q(a).$$
So if the ground field is algebraically closed, any matrix $A$ is similar to a Jordan normal form, and then we may use Theorem \ref{thm2} to construct the corresponding semi-invariant generators for $A$.

\section{Relation with invariants and covariants of binary forms}

This section spells out the relation between our work, with the classical invariant theory of 19th century from the classical book by Hilbert \cite{Hil}. First, we recall the definition of an invariant polynomial ${\mathcal I}(a_0,\dots, a_n)$ of $n+1$ variables of the binary form
$$f(X_1,X_2)=\sum_{i=0}^n \frac{a_i}{(n-i)!}X_1^{n-i}X_2^i.$$
If a linear change of variables $X_1=\alpha_{11}X_1'+\alpha_{12}X_2'$ and $X_2=\alpha_{21}X_1'+\alpha_{22}X_2'$ with non-zero determinant $\delta=\alpha_{11}\alpha_{22}-\alpha_{12}\alpha_{21}$ is performed on $f$, it will become
$$f'(X',Y')=f(\alpha_{11}X_1'+\alpha_{12}X_2', \alpha_{21}X_1'+\alpha_{22}X_2')=\sum_{i=0}^n \frac{a_i'}{(n-i)!}{X_1'}^{n-i}{X_2}'^i.$$
We say that ${\mathcal I}(a_0,\dots, a_n)$ is an invariant form of weight $g$ if
$${\mathcal I}(a_0',\dots, a_n')=\delta^g {\mathcal I}(a_0,\dots, a_n).$$
Similarly a covariant polynomial ${\mathcal C}(a_0,\dots, a_n,X_1,X_2)$ of weight $g$ is a form such that 
$${\mathcal C}(a_0',\dots, a_n',X_1',X_2')=\delta^g {\mathcal C}(a_0,\dots, a_n,X_1,X_2).$$

The simplest covariant is $f$ itself, which is a covariant of weight $0$. Any invariant of weight $g$, is a covariant that has no $X_1$ or $X_2$ in it. Consider the generic polynomial 
$${\mathcal C}(a_0,\dots,a_n,X_1,X_2)=\sum_{i=0}^m C_i(a_0,\dots, a_n)X_1^{m-i}X_2^i.$$
It is easy to see that the weight $g$ covariance property for $\mathcal C$ only for the diagonal change of variables $X_1=\kappa X_1'$ and $X_2=\lambda X_2'$, is equivalent to the condition that all $C_i$'s are homogeneous of the same degree and each term of $C_i$ is of weight $g+i$. In particular an invariant form of weight $g$ is homogeneous of weight $g$. The covariance property for the matrices of the form $\left[\begin{array}{cc}1&\lambda\\0&1\end{array}\right]$ amounts to 
$$D{\mathcal C}=X_2\frac{\partial {\mathcal C}}{\partial X_1}.$$
Here as before
$$D=\sum_{i=1}^n a_{i-1}\frac{\partial}{\partial a_i}.$$
In particular $DC_0=0$, so for any invariant $\mathcal I$, we have $D{\mathcal I}=0$. Finally the covariance property for the matrices of the form $\left[\begin{array}{cc}1&0\\\lambda&1\end{array}\right]$ amounts to 
$$\Delta_n {\mathcal C}=X_1\frac{\partial {\mathcal C}}{\partial X_2}.$$
Here the differential operator $\Delta$ is given by
$$\Delta_n=\sum_{i=0}^{n-1} (n-i)(i+1)a_{i+1}\frac{\partial}{\partial a_i}.$$
In particular $\Delta C_m=0$, so for any invariant $\mathcal I$, we have $\Delta{\mathcal I}=0$. Since these three classes of functions generate all matrices by multiplication, we see that these three conditions are equivalent to the covariance property and hence the invariance property. Conversely, let $C_0$ be an isobaric homogeneous form in terms of $a_0,\dots, a_n$ of degree $d$ and weight $g$, with $DC_0=0$. Then there is a unique covariant of the form
$${\mathcal C}= C_0X_1^m+C_1X_1^{m-1}X_2+\dots +C_mX_2^m$$
where $m$ is the smallest non-negative integer such that $\Delta_n^{m+1}C_0=0$ (which happens to be $nd-2g$) . In fact from 
$$\Delta {\mathcal C}=X_1\frac{\partial C}{\partial X_2}$$
it follows that $C_i=i\Delta C_{i-1}$ for $i=1,\dots, m$ and hence
$$C_i=\frac{\Delta^i C_0}{i!}$$
On the other hand, it is not hard to show that with this definition $D {\mathcal C}=X_2\frac{\partial {\mathcal C}}{\partial X_1}$. So $\mathcal C$ is the unique covariant with $X_1^m$ coefficient equals to $C_0$. One calls $C_0$ the source for $\mathcal C$.
 One can easily show that if ${\mathcal A}$ is a homogeneous and isobaric form of degree $d$ and weight $g$ in $K[a_0,\dots, a_n]$, then 
 \begin{equation}\label{com}
(D\Delta_n-\Delta_n D)({\mathcal A})=(nd-2g){\mathcal A}.
\end{equation}
This follows by evaluating both sides on $a_i$ and noticing that both sides are derivations. An easy induction on $k$ implies the following more general commutator relation. 
\begin{equation}\label{com2}
(D^k\Delta_n -\Delta_n D^k)({\mathcal A})=k(nd-2g+k-1)D^{k-1}({\mathcal A}).
\end{equation}
\section{On the relation between semi-invariants and U invariants}

We have learned the following beautiful piece of 19th century mathematics from Kraft and Procesi \cite{KP}. The space of weight $g$ and degree $d$ homogeneous isobaric polynomials in $K[a_0,a_1,\dots]$ is denoted by $K[a]_{d,g}$. It is a finite dimensional space with a basis 
$$a_{h_1}\dots a_{h_d}$$
where $h_1\ge h_2\ge \dots\ge h_d\ge 0$ and $h_1+\dots+h_d=g$. Its dimension is the number of partitions of $g$ which is equal to the number of $(k_1,\dots,k_d)$ with $k_1\ge 0, \dots, k_d\ge 0$ and $k_1+2k_2+\dots +dk_d=g$. Now we construct a different basis for $K[a]_{d,g}$ as
$$U_{k_1,\dots, k_d}$$
indexed by such $(k_1,\dots, k_d)$. This is defined via the following slick method, that is due to Stroh from an 1890 paper \cite{S}. For variables $\lambda_1,\dots, \lambda_d$ define\
\begin{equation}\label{form0}
\pi_{d,g}=\sum_{\substack{h_1\ge 0,\dots, h_d\ge 0\\ h_1+\dots+h_d=g}} \lambda_1^{h_1}\dots\lambda_d^{h_d}a_{h_1}\dots a_{h_d}.
\end{equation}
Note that we have not imposed the condition of $h_1\ge \dots\ge h_d$, hence although $a_{h_1}\dots a_{h_d}$ is invariant under the permutation of $h_i$'s, its coefficient $\lambda_1^{h_1}\dots \lambda_d^{h_d}$ is not. Collecting its orbit under the action of the permutation group under one function that we call $m_{h_1,\dots, h_n}$ we arrive at the formula
\begin{equation}\label{form1}
\pi_{n,g}=\sum_{\substack{h_1\ge \dots\ge h_n\ge 0\\ h_1+\dots+h_n=g}} m_{h_1,\dots, h_n}a_{h_1}\dots a_{h_n}.
\end{equation}
Note that $m_{h_1,\dots, h_n}$ for all $h_1\ge \dots\ge h_d\ge 0$ and $h_1+\dots+h_d=g$ is a basis for the space of all symmetric homogeneous polynomials of degree $g$ and variables $\lambda_1,\dots, \lambda_d$. Another basis for this space is given using the symmetric elementary polynomials
$$e_k=\sum_{0\le i_1<\dots<i_k\le d} \lambda_{i_1}\dots\lambda_{i_k}$$
for $k=1,\dots, d$. This basis is given by $e_1^{k_1}\dots e_d^{k_d}$ with $k_i\ge 0$ and $\sum_{i=1}^d ik_i=g$. So we may rewrite $\pi_{d,g}$ as
\begin{equation}\label{form2}
\pi_{d,g}=\sum_{\substack{k_1\ge 0,\dots, k_d\ge 0\\ k_1+2k_2+\dots+dk_d=g}} e_1^{k_1}\dots e_d^{k_d} U_{k_1,\dots, k_d}.
\end{equation}

We can write $U_{k_1,\dots, k_d}$ more explicitly as follows. Define the change of basis coefficients $\alpha_{h_1,\dots, h_d; k_1,\dots, k_d}$ as
$$m_{h_1\dots,h_d}=\sum \alpha_{h_1,\dots, h_d; k_1,\dots, k_d}e_1^{k_1}\dots e_d^{k_d}$$
where the sum is taken over all non-negative $k_i$'s with $\sum_{i=1}^d ik_i=\sum_{i=1}^d h_i$.

\begin{lemma}
The polynomials $U_{k_1,\dots, k_d}$ for $k_i\ge 0$ and $\sum_{i=1}^d ik_d=g$, form a basis for the space $K[a]_{d,g}$ of homogeneous polynomials of weight $g$ and degree $d$ in terms of $a_0,a_1,\dots$.
\end{lemma}
\begin{proof} We recall the proof given in \cite{KP}. A duality tensor in $V\otimes W$, where $V$ and $W$ are isomorphic finite dimensional vector spaces is a tensor of the form $\pi=\sum_{i=1}^n a_i\otimes b_i$. where $a_1,\dots, a_n$ is a basis for $V$ and $b_1,\dots, b_n$ is a basis for $W$. It is easy and standard to see if we rewrite $\pi=\sum_{i=1}^n a_i'\otimes b_i'$ where $a_1',\dots, a_n'$ form a basis for $V$ then $b_1',\dots, b_n'$ form a basis for $W$. This is essentially equivalent to the fact that if a linear map sends one basis to a basis, then it sends any basis to another basis. Now use this result to two forms of $\pi_{d,g}$ given by equations \ref{form1} and \ref{form2}.
\end{proof}

Now why do we take so much trouble to come up with a new basis for $K[a]_{d,g}$, while we have a very simple basis $a_{h_1}\dots a_{h_d}$ with $h_1\ge \dots \ge h_d\ge 0$ and $h_1+\dots+h_d=g$? The new basis, $U_{k_1,\dots, k_d}$ has a very interesting and amazing property with respect to the action of the operator $D$ and the matrix $J$. It is summarized in the following theorem. The first part is in \cite{KP} and the second part is an easy extension. 

\begin{theorem} \label{DJ}
One has the following properties, where by convention we let $U_{k_1,\dots, k_d}=0$ is one of $k_i$'s is negative.
\begin{enumerate}
\item One has $DU_{k_1,k_2,\dots, k_d}=U_{k_1-1,k_2,\dots, k_d}$, hence the set of $U_{0,k_2,\dots, k_d}$ with $k_2\ge 0,\dots, k_d\ge 0$ and $\sum_{i=2}^d ik_i=g$ forms a basis for the space of homogeneous $U$ invariants of degree $d$ and weight $g$.
\item The polynomial $U_{k_1,\dots, k_d}(Ja)$ is equal to
$$U_{k_1,\dots, k_d}(a)+\sum_{i=1}^d U_{k_1,\dots, k_i-1,\dots, k_d}(a).$$
\end{enumerate}
\end{theorem}
\begin{proof} Again we recall the beautiful proof given in \cite{KP}. Let us assume $a_{-1}=0$. Note that from equation \ref{form0} one has
$$\lambda_i \pi_{d,g-1}=\sum_{\substack{h_1\ge 0,\dots, h_d\ge 0\\h_1+\dots+h_d=g}} \lambda_1^{h_1}\dots\lambda_i^{h_i}\dots \lambda_d^{h_d}a_{h_1}\dots a_{h_i-1}\dots a_{h_d}$$
And therefore if we sum over all $i=1,\dots, d$ we arrive at
$$e_1\pi_{d,g-1}=D\pi_{d,g}$$
comparing the coefficient of $e_1^{k_1}\dots e_d^{k_d}$ of both sides we arrive at part one. Part two is similarly obtained once we realize that for $1\le i_1<\dots i_k\le d$
$$\lambda_{i_1}\dots \lambda_{i_k}\pi_{d,g-k}=\sum_{\substack{h_1\ge 0,\dots, h_d\ge 0\\h_1+\dots+h_d=g}} \lambda_1^{h_1}\dots \lambda_d^{h_d}a_{h_1}\dots a_{h_{i_1}-1}\dots a_{h_{i_k}-1}\dots a_{h_d}$$ and hence
$$\pi_{d,g}(Ja)=\pi_{d,g}+e_1\pi_{d,g-1}+e_2\pi_{d,g-2}+\dots+e_d\pi_{d,g-d}.$$
This implies part 2 if we compare the coefficient of $e_1^{k_1}\dots e_d^{k_d}$ of both sides.
\end{proof}

The following classical result follows immediately from part one of the theorem above that gives a basis for the space of $U$ invariants of weight $g$ and degree $d$.
 \begin{corollary}
 The dimension of the space of $U$ invariants in $K[a_0,a_1.\dots]$ which are homogeneous of degree $d$ and of weight $g$ is the coefficient of $x^g$ in the power series expansion of 
 $$\frac{1}{(1-x^2)(1-x^3)\dots (1-x^d)}.$$
 \end{corollary}
 
 There are some recursive relations among $U_{k_1,\dots, k_d}$ that we give a few of them here.
 \begin{lemma}
 One always has 
 $$U_{k_1,\dots, k_d,0}=a_0U_{k_1,\dots, k_d}.$$
 \end{lemma}
 \begin{proof}
 This immediately follows from observing that the coefficient of $e_1^{k_1}\dots e_d^{k_d}$ in $m_{h_1,\dots, h_d,h_{d+1}}$ is zero if $h_1\ge\dots\ge h_{d+1}>0$, since it is divisible by $e_{d+1}$. Also, if $h_{d+1}=0$, this coefficient is the same as the coefficient of it in the $d$ variable symmetric polynomial $m_{h_1,\dots, h_d}$. This follows if we let $\lambda_{d+1}=0$. 
 \end{proof}
 
 \begin{lemma}
 One always has
 $$U_{k_1,\dots, k_d,1}=a_1U_{k_1,\dots, k_{d-1},k_d+1}-a_0(k_1+1)U_{k_1+1,k_2,\dots, k_{d-1},k_d+1}-a_0(k_d+2)U_{k_1,\dots, k_{d-2},k_{d-1}-1,k_d+2}-$$
 $$-a_0\sum_{i=2}^{d-1} (k_{i}+1)U_{k_1,\dots, k_{i-1}-1, k_{i}+1,\dots, k_{d-2},k_d+1}$$
 \end{lemma}
 \begin{proof} Let $e_0=1$, $e_1,e_2,\dots, e_{d+1}$ be the elementary symmetric polynomials in $\lambda_1,\dots, \lambda_{d+1}$. Our key observation is that for $i=1,\dots, d+1$ we have
 $$\frac{\partial}{\partial \lambda_{d+1}} e_i=e_{i-1}|_{\lambda_{d+1}=0}.$$
 Now
 $$\left.\frac{\partial}{\partial \lambda_{d+1}}\right|_{\lambda_{d+1}=0}\pi_{d+1,g}=a_1\pi_{d,g-1}$$
 since the only terms that contribute non-zero terms are those with $h_{d+1}=1$. Now by product rule and chain rule, we have
  
 $$\left.\frac{\partial}{\partial \lambda_{d+1}}\right|_{\lambda_d=0} e_1^{k_1}\dots e_{d+1}^{k_{d+1}}=0\quad\mbox{if}\quad k_{d+1}>1$$
 $$\left.\frac{\partial}{\partial \lambda_{d+1}}\right|_{\lambda_d=0} e_1^{k_1}\dots e_{d}^{k_{d}}e_{d+1}=e_1^{k_1}\dots e_d^{k_d+1}$$
  $$\left.\frac{\partial}{\partial \lambda_{d+1}}\right|_{\lambda_d=0} e_1^{k_1}\dots e_{d}^{k_d}=k_1e_1^{k_1-1}e_2^{k_2}\dots e_d^{k_d}+\sum_{i=2}^d k_ie_1^{k_1}\dots e_{i-1}^{k_{i-1}+1}e_i^{k_i-1}\dots e_d^{k_d}$$  
If we use these formulae, and take evaluate 

 $$\left.\frac{\partial}{\partial \lambda_{d+1}}\right|_{\lambda_{d+1}=0}\pi_{d+1,g}$$
 we get
 $$\sum_{k_1+\dots+(d-1)k_{d-1}+d(k_d+1)=g-1} e_1^{k_1}\dots e_d^{k_d+1}U_{k_1,\dots, k_d,1}+\sum k_1e_1^{k_1-1}\dots e_d^{k_d}U_{k_1,\dots, k_d,0}+$$
 $$+\sum_{k_1+2k_2+\dots+dk_d=g} \sum_{i=2}^{d} k_i e_1^{k_1}\dots e_{i-1}^{k_{i-1}+1}e_i^{k_i-1}\dots e_d^{k_d}U_{k_1,\dots, k_d,0}$$ 
   
 Now, the lemma will be proved if we compare the coefficients of $e_1^{k_1}\dots e_{k_{d-1}}^{k_{d-1}}e_d^{k_d+1}$ of both sides of the identity
 $$\left.\frac{\partial}{\partial \lambda_{d+1}}\right|_{\lambda_{d+1}=0}\pi_{d+1,g}=a_1\pi_{d,g-1}.$$  
 On the left hand side it is
 $$U_{k_1,\dots, k_d,1}+(k_1+1)U_{k_1+1,k_2,\dots, k_{d-1},k_d+1,0}+\sum_{i=2}^{d-1} (k_i+1)U_{k_1,\dots, k_{i-1}-1, k_{i}+1,\dots, k_{d-2},k_d+1,0}+$$  
 $$+(k_d+2)U_{k_1,\dots, k_{d-1}-1,k_d+2,0}$$
 On the right hand side it is
 $$a_1U_{k_1,\dots, k_{d-1},k_d+1}.$$
 So the lemma is proved.
  \end{proof}

 We collect the following observation about the relation between the polynomials $p_i$ of section 2 with the basis $U_{k_1,\dots, k_d}$.
 \begin{theorem} 
 We have the following two cases.
 \begin{enumerate}
\item If $n=2k$ then
$$(-1)^kp_{n}=U_{0,k}-U_{1,k-1}+\dots \pm U_{k,0}=\sum_{i=0}^k (-1)^iU_{i,k-i}.$$
 \item If $n=2k+1$ then the weight $n$ part of $(-1)^k p_n$ is $U_{0,k-1,1}$. In fact
 $$(-1)^k p_n= \sum_{i=0}^{k-1} (-1)^i\left(U_{i,k-i-1,1}- (k-i)U_{i,k-i,0}\right).$$
\end{enumerate}
 \end{theorem}
 \begin{proof}
 The case $n=2k$ is simple and is left to the reader. For the case when $n=2k+1$, one knows from remark \ref{g} that
 $$(-1)^kp_n=(-1)^ka_1p_{2k}-(-1)^{k+1}(W(p_{2k+2})-(2k+2)p_{2k+2})$$
 hence by the first part we have:
 $$(-1)^kp_n=\sum_{i=0}^k (-1)^ia_1U_{i,k-i}-\sum_{i=0}^k (-1)^{i}(i+1)a_0U_{i+1,k-i}$$
 So if we compare the parts of the same weight with the formula that we need to show, we see that it is enough to show that
 $$ a_1U_{i,k-i}-(i+1)a_0U_{i+1,k-i}=U_{i,k-i-1,1}+(k-i+1)U_{i-1,k-i+1,0}$$
 If we let $j=k-i-1$, this can be rewrite this as follows.
 $$U_{i,j,1}=a_1U_{i,j+1}-(i+1)a_0U_{i+1,j+1}-(j+2)U_{i-1,j+2,0}$$
 This is a special case of Lemma 4.3. So the lemma is proved.
 
 \end{proof}
 \begin{remark}
As an example of part 2 of the theorem above, we collect the following examples.
 $$-p_3=U_{0,0,1}-U_{0,1,0}$$
 $$p_5=U_{0,1,1}-U_{1,0,1}-2U_{0,2,0}+U_{1,1,0}$$
 $$-p_7=U_{0,2,1}-U_{1,1,1}-3U_{0,3,0}+U_{2,0,1}+2U_{1,2,0}-U_{2,10}$$
 \end{remark}
 
 Finally, we address the following question.
 \\
 \\
 {\textbf{Question.}} Given a homogeneous $U$-invariant $q(a)\in K[a_0,a_1,\dots]$ of degree $d$ and weight $g$, is it possible to add homogeneous polynomials of degree $d$ and weight less than $g$ to $q$ to make it into a $J$-invariant polynomial? We call this process $J$-completion.
 \\
 \\
 We will answer this question affirmatively, but this completion is by no means unique. Since $U_{0,k_2,\dots, k_d}$ with $\sum_{i=2}^d ik_i=g$ form a basis for the space of homogeneous $U$-invariants of degree $d$ and weight $g$, it is enough to exhibit this completion only for these elements. The following theorem gives an explicit way to do so.
 \begin{theorem}\label{comp1}
 Given $U_{0,k_2,\dots, k_d}$ with $k_i\ge 0$ and $k=\sum_{i=2}^d k_i$, then 
 $$U_{0,k_2,\dots, k_d}+\sum_{i=1}^k \sum_{\substack{0\le l_2\le k_2,\dots, 0\le l_d\le k_d\\
 l_2+\dots+l_d=k-i}} \frac{(-1)^i i!}{(k_2-l_2)!\dots (k_d-l_d)!} U_{i,l_1,\dots, l_d}$$
 is a $J$-completion of $U_{0,k_2,\dots, k_d}$.
 \end{theorem}
 \begin{proof}
 Assume that the weight of $U_{0,k_1,\dots, k_d}$ is $g$, where $g=\sum_{i=2}^d ik_i$. Note that the weight of $U_{i,l_2,\dots, l_d}$ if $l_j\le k_j$ and $\sum_{j=2}^d l_j=k-i$ is at most $i+2(k_2-i)+3k_3+\dots +dk_d=g-i$ so the terms added are on weights strictly less than $g$. Call the sum above $S(a)$, then by part (2) of theorem \ref{DJ}, one has $S(Ja)-S(a)$ equals to 
 $$U_{0,k_2-1,k_3,\dots, k_d}+U_{0,k_2,k_3-1,\dots, k_d}+\dots+U_{0,k_1,\dots, k_{d-1},k_d-1}$$
 $$\sum_{i=1}^k \sum_{l_2+\dots+l_d=k-i}\frac{(-1)^i i!}{(k_2-l_2)!\dots(k_d-l_d)!}(U_{i-1, l_2,\dots, l_d}+U_{i,l_2-1,\dots, l_d}+\dots+U_{0,l_2,\dots, l_d-1})$$
 The terms on the first line can be combined with the sum if we let $i=0$ to $k$. Note that 
 $$i=(k_2-l_2)+\dots+(k_d-l_d)$$
 since $l_2+\dots+l_d=k-i$ and $k_2+\dots+k_d=k$. So these terms cancel each other telescopically by the following identity
 $$\frac{i!}{(k_2-l_2)!\dots (k_d-i_d)!}=\frac{(i-1)!}{(k_2-l_2-1)!\dots (k_d-l_d)!}+\frac{(i-1)!}{(k_2-l_2)!(k_3-l_3-1)!\dots (k_d-l_d)!}+\dots $$
 $$+\frac{(i-1)!}{(k_2-l_2)!\dots (k_d-l_d-1)!}$$
 and hence $S(Ja)=S(a)$ as desired.
 
 \end{proof}
 
 \begin{remark}
 For example if we apply this theorem to $U_{0,k-1,1}$ we get
 $$\sum_{i=0}^{k-1} (-1)^i (U_{i,k-1-i,1}- (i+1) U_{i+1,k-1-i,0}).$$
 \end{remark}

 We now provide another method to construct this completion, that does not need the basis $U_{0,k_2,\dots, k_d}$ for the space of $U$-invariants. First, define the following operator on $K[a_0, a_1,\dots]$.
$$L=\sum_{i=0}^{\infty} (i+1)a_{i+1}\frac{\partial}{\partial a_i}.$$
\begin{lemma}
One has $DL-LD=id$. More generally for any $k\ge 1$, one has $$D^kL-LD^k=kD^{k-1}.$$
\end{lemma}\label{comm}
\begin{proof}
Since both sides of equality are derivations on $K[a_0,a_1,\dots]$, it is enough to check the equality for $a_i$. Now for $i>0$
$$(DL-LD)(a_i)=(i+1)a_i-ia_i=a_i$$
for $i=0$
$$(DL-LD)(a_0)=a_0-0=a_0.$$
The last claim follows inductively as follows
$$D^kL-LD^k=D^{k-1}(LD+1)-LD^k=(D^{k-1}L-LD^{k-1})D+D^{k-1}=$$
$$=(k-1)D^{k-1}+D^{k-1}=kD^{k-1}$$
\end{proof}

For the next lemma, we need Stirling's number of the first kind $n\brack k$ that are defined as follows.
For $n>0$, one has ${n\brack 0}={0\brack n}=0$ and ${0\brack 0}=1$. We have the following recursive formula
$${n+1\brack k+1}=n{n\brack k+1}+{n\brack k}.$$

\begin{lemma}
If $p\in K[a_0,a_1,\dots]$ and $D^mp=0$ then 
$$D\left(\frac{1}{m!}\sum_{i=1}^m (-1)^{i+1}{m+1\brack i+1}(LD)^{i-1}L\right) (p)=p.$$
So the operator 
$$D^{-1}:=\frac{1}{m!}\sum_{i=1}^m (-1)^{i+1}{m+1\brack i+1}(LD)^{i-1}L$$
is a right inverse for the operator $D$ on the subspace of polynomials with $D^mp=0$. 
\end{lemma}

\begin{proof} If $Dp=0$ then by lemma \ref{comm}, it follows that $DLp=p$. So the claim is true for $m=1$. If $m>1$ and $D^mp=0$, then by lemma \ref{comm}, $D^mLp=mD^{m-1}p$ and hence 
$q=DL(p)-mp$ has the property that $D^{m-1}q=0$. Now by the induction hypothesis
$$D\left(\frac{1}{(m-1)!}\sum_{i=1}^{m-1} (-1)^{i+1}{m\brack i+1}(LD)^{i-1}L\right)q=q.$$
If we let $q=DL(p)-mp$, we get
$$D\left(\frac{1}{(m-1)!}\left( \sum_{i=1}^{m-1} (-1)^{i+1}{m\brack i+1} (LD)^i L-\sum_{i=1}^{m-1} (-1)^{i+1}m{m\brack i+1}(LD)^{i-1}L \right)\right) p=DL(p)-mp$$
Since ${m\brack m+1}=0$, we rewrite it as
$$D\left(\frac{1}{(m-1)!}\left( -\sum_{i=2}^{m} (-1)^{i+1}{m\brack i} (LD)^{i-1} L- \sum_{i=1}^{m} (-1)^{i+1}m{m\brack i+1}(LD)^{i-1}L \right)\right) p=DL(p)-mp$$

Since ${m\brack 1}=(m-1)!$, using the recursive relation we get, 

$$D\left(L-\frac{1}{(m-1)!}\sum_{i=1}^m (-1)^{i+1}{m+1\brack i}(LD)^{i-1}L\right) (p)=DL(p)-mp.$$
This simplifies to 
$$D\left(\frac{1}{m!}\sum_{i=1}^m (-1)^{i+1}{m+1\brack i}(LD)^{i-1}L\right) (p)=p.$$

\end{proof}
\begin{remark} Since any polynomial $p\in K[a_0,a_1,\dots]$ has the property that $D^{w+1}p=0$ where $w$ is the maximum weight of $p$, hence the previous lemma defines an inverse for $D$. This inverse is not unique. Using the basis $U_{k_1,\dots,k_d}$ another inverse can be simply defined by $D^{-1}(U_{k_1,\dots, k_d})=U_{k_1+1,k_2,\dots, k_d}$.
\end{remark}

\begin{theorem}\label{comp2}
If $q_0$ is a homogeneous $U$ invariant of degree $d$ and weight $w$, for $i=1,\dots, w-1$ define
$$q_i(a)=q_{i-1}(a)-D^{-1}([q_{i-1}(J a)-q_{i-1}(a)]_{w-1-i}).$$
where $[q]_g$ denotes the weight $g$ component of $q$.
Then $q_{w-1}$ is a $J$-completion of $q_0$.
\end{theorem}
\begin{proof}
Inductively, we show that the highest weight of $q_i(Ja)-q_i(a)$ is $w-i-2$. This shows that for $i=w-1$, $q_i(Ja)=q_i(a)$. In the proof of lemma \ref{JJ}, we showed that if $p$ is of weight at most $w$ then 
$$p(Ja)=p(a)+Dp(a)+\mbox{terms of weight less than or equal}\:\: w-2$$
So the claim is clear for $q_0$, since $Dq_0=0$. Now by the induction hypothesis
$$q_{i-1}(Ja)=q_{i-1}(a)+[q_{i-1}(Ja)-q_{i-1}(a)]_{w-1-i}+\mbox{terms of weight at most}\:\: w-2-i$$
And by the formula above about the relation between $D$ and $J$ and the fact that $D^{-1}$ is an inverse for $D$ we have
$$D^{-1}([q_{i-1}(Ja)-q_{i-1}(a)]_{w-1-i})(Ja)=$$
$$D^{-1}([q_{i-1}(Ja)-q_{i-1}]_{w-1-i}]+[q_{i-1}(Ja)-q_{i-1}(a)]_{w-1-i}+\mbox{terms of weights at most}\:\:\: w-2-i$$
If we subtract these two equations we find that
$$q_i(Ja)=q_i(a)+\mbox{terms of weights at most}\:\: w-2-i$$
and the theorem is proved.
\end{proof}

\begin{remark} Both methods of theorem \ref{comp1} and theorem \ref{comp2} given a homogeneous $U$ invariant $q_0$ of degree $d$ and weight $w$, produce homogeneous $J$-invariants of degree $d$ and top weight component $q_0$, however if $q_0$ is in variables $a_0,\dots, a_n$, the completion might have variables in $a_0,\dots, a_w$. So if $w>n$, then it is plausible to have new variables in the completion. For example, we apply these methods to the $U$ invariant of degree $4$ and weight $6$
$$q_0(a_0,a_1,a_2,a_3)=-3a_2^2a_1^2+6a_3a_1^3+8a_2^3a_0-18a_3a_2a_1a_0+9a_3^2a_0^2$$
of Example \ref{ex5}.
\\
If we want to use theorem \ref{comp1}, we need to express $q_0$ in terms of the basis $U_{k_1,\dots, k_d}$. It is as follows.
$$q_0=2U_{0,0,2,0}-3U_{0,1,0,1}-6U_{0,3,0,0}.$$
Now an application of theorem \ref{comp1} yields the following $J$-completion.
$$q_0+(6U_{1,2,0,0}+3U_{1,0,0,1})-(2U_{1,0,1,0}+6U_{2,1,0,0})+(6U_{3,0,0,0}+3U_{1,1,0,0})-4U_{2,0,0,0}.$$
After computing these basis terms by Mathematica, we discover that this completion has variables, $a_0,\dots, a_5$.
If we use instead, the algorithm in theorem \ref{comp2}, we get the following $J$-completion that has variables $a_0,\dots,a_4$.
$$q_0+(3a_2a_1^3-6a_2^2a_1a_0-9a_3a_1^2a_0+18 a_3a_2a_0^2)$$
$$+\frac{1}{8}(-7a_1^4-12a_2a_1^2a_0+30 a_2^2a_0^2+36a_3a_1a_0^2-12a_4a_0^3)
-\frac{3}{2}(-a_1^3a_0+2a_3a_0^3)
-\frac{5}{16}(3a_1^2a_0^2+2a_2a_0^3)$$
$$+\frac{5}{16}a_1a_0^3.$$
\end{remark}
It is interesting to see if there is a method for $J$-completion, that does not introduce new variables. Indeed, this is possible. 
\begin{lemma}
Let $N=J-I$ and $J'=\exp(N)$ then $J=S^{-1}J'S$, where $S$ is a lower triangular matrix with $S_{ii}=1$ and 
$$S_{ij}=\frac{1}{(i-1)!}\sum_{k=0}^j (-1)^k{j-1\choose k}{(j-1-k)}^{i-1}$$
\end{lemma}
\begin{proof}
If $e_1,\dots, e_n$ is the standard basis for $K^n$, then $Ne_i=e_{i+1}$, with $e_{n+1}$ is assumed to be zero. Now let $N'=J'-I$, we compute $(N')^ke_1$ for $k=1,\dots, {n-1}$. Since $N^m=0$ for $m\ge n$, hence $(N')^{n-1}=N^{n-1}$ and hence $(N')^{n-1}e_1=e_n\ne 0$. If we let $e_i'=(N')^{i-1}e_1$ then it is easy to see that they form a new basis for $K^n$ and $N'e_i'=e_{i+1}'$, with $e_{n+1}'$ is assumed to be zero. This shows that $J$ and $J'$ are similar. To find the matrix $S$, we need to compute $e_j'$ in terms of the standard basis. The coefficient of $e_i$ in the expansion of $e_j'$ s
$$S_{ij}=\sum_{\substack{k_1+\dots+k_{j-1}=i-1\\k_1,\dots, k_{j-1}>0}} \frac{1}{k_1!\dots k_{j-1}!} .$$
Note that when $i=j$, then $S_{ii}=1$. Also $(i-1)!S_{ij}$ is the number of ways to partition a set with $i-1$ elements into $j-1$ non-empty labeled subsets. By inclusion-exclusion principle this number is 
$$\sum_{k=0}^{j-1} (-1)^k {j-1\choose k} (j-1-k)^{i-1}.$$
Finally we remark that $S_{ij}=\frac{(j-1)!}{(i-1)!}S(i-1,j-1)$, where $S(i-1,j-1)$ is the Stirling's number of the second kind.
\end{proof}
\begin{lemma}\label{UJ}
 With the notation of the previous lemma, a form is a $U$-invariant if and only if it a $J'$-invariant. Therefore $p(a)$ is $U$-invariant if and only if $p(Sa)$ is a $J$-invariant, i.e., a semi-invariant.
\end{lemma} 
\begin{proof}
This follows from two facts. First, for the nilpotent operator $D$, one knows that $Dp=0$ if and only if $\exp(D)p=p$ and second, 
$$\exp(D)p(a)=p(J'a).$$
First assertion is standard and the second follows from either the Taylor's expansion or checking it for a monomial. The relation given in the lemma between $U$-invariants and $J$-invariants follows from the discussion at the end of section 2, about semi-invariants of two similar matrices.
\end{proof}
\begin{remark} Lemma \ref{UJ} gives a method of $J$-completion without introducing new variables. This is because if $p(a)$ is an isobaric $U$-invariant then $p(Sa)$ is a $J$-invariant that is of the form $p(a)$ plus terms of lower weights. If we apply this method to
$$q_0(a_0,a_1,a_2,a_3)=-3a_2^2a_1^2+6a_3a_1^3+8a_2^3a_0-18a_3a_2a_1a_0+9a_3^2a_0^2,$$
we get the following $J$-completion.
\begin{eqnarray}\nonumber
q_0-\frac{1}{108}\hspace{-0.3cm}&(&\hspace{-0.3cm}(-1944a_0^2a_2a_3+648a_0a_1a_2^2+972a_0a_1^2a_3-324a_1^3a_2)\\\nonumber
&+&\hspace{-0.3cm}(-540a_0^2a_2^2-324a_0^2a_1a_3+216a_0a_1^2a_2+81a_1^4)+(108a_0^2a_1a_2-162a_0a_1^3)\\\nonumber
&+&\hspace{-0.3cm}(117a_0^2a_1^2-72a_0^3a_2)+(-36a_0^3a_1)+(4a_0^4))\nonumber
\end{eqnarray}
We were not able to find a general method of $J$-completion that gives the form in Example \ref{ex5}.

\end{remark}

{\textbf{Acknowledgment.}} The authors wish to thank heartily Claudio Procesi, whose kind response to some of our questions, improved the presentation and quality of the paper a lot.



\end{document}